\documentclass[11pt]{amsart}
\usepackage{amscd}
\usepackage{amsmath}
\usepackage{amssymb}
\usepackage{latexsym}
\usepackage{amsthm}
\usepackage{hyperref}

\newtheorem{theorem}{Theorem}[section]
\newtheorem{lemma}[theorem]{Lemma}

\theoremstyle{theorem}
\newtheorem{definition}[theorem]{Definition}

\newtheorem{proposition}[theorem]{Proposition}
\newtheorem{corollary}[theorem]{Corollary}

\theoremstyle{remark}

\numberwithin{equation}{section}

\def\Q{\mathbb{Q}}
\def\C{\mathbb{C}}

\def\Z{\mathbb{Z}}

\def\Y{Y}

\def\a{\alpha}

\def\la{\lambda}
\def\La{\Lambda}

\def\T{\widetilde{T}}
\def\B{\widetilde{B}}
\def\G{\widetilde{G}}
\def\M{\widetilde{M}}

\def\U{\widetilde{U}}

\def\H{\widetilde{H}}

\def\Hom{\,\mbox{Hom}\,}
\def\Ind{\,\mbox{Ind}}

\def\End{\,\mbox{End}\,}
\def\Aut{\,\mbox{Aut}\,}
\def\Rep{\,\mbox{Rep}\,}

\def\Spec{\,\mbox{Spec}}

\newcommand{\map}[2]{\,{:}\,#1\!\longrightarrow\!#2}

\def\ww{W_\chi} 
\def\h{\mathcal{H}(\G,K)} 
\def\hh{\mathcal{H}(\G,I)} 

\def\p{\varpi} 

\def\TT{\bf{T}}

\def\lqt{\backslash} 

\numberwithin{equation}{section}

\title[Representations of Metaplectic Groups]{Principal Series Representations of Metaplectic Groups Over Local Fields}
\address{Department of Mathematics, Massachusetts Institute of Technology, MA 02139}\email{petermc@math.stanford.edu}
\author{Peter J McNamara}
\date{August 2010}

\begin{document}
\maketitle
\begin{abstract}Let $G$ be a split reductive algebraic group over a non-archimedean local field. We study the representation theory of a central extension $\G$ of $G$ by a cyclic group of order $n$, under some mild tameness assumptions on $n$. In particular, we focus our attention on the development of the theory of principal series representations for $\G$ and its applications to the study of Hecke algebras via a Satake isomorphism.
\end{abstract}

\section{Introduction}

Let $F$ be a non-archimedean local field with ring of integers $O_F$ and assume that $G$ is a split reductive group over $F$ that arises by base extension from a smooth reductive group scheme $\bf{G}$ over $O_F$. Let $n$ be a positive integer such that $2n$ is coprime to the residue characteristic of $F$ and that $F^\times$ contains $2n$ distinct $2n$-th roots of unity. The object of this paper is to study the principal series representations of a group $\G$ which arises as a central extension of (the $F$ points of) $G$ by the cyclic group $\mu_n$ of order $n$. This means that there is an exact sequence of topological groups
\[
1\rightarrow \mu_n\rightarrow \G\rightarrow G\rightarrow 1
\]
with the kernel $\mu_n$ lying inside the centre of $\G$.

This metaplectic group $\G$ is a locally compact, totally disconnected topological group. Following the example of reductive case, we study the simplest family of representations, namely those which are induced from the inverse image in $\G$ of a Borel subgroup of $G$. Such representations have been studied in the literature for particular classes of groups. Kazhdan and Patterson \cite{kp} have a detailed study in the case of $G=GL_n$, while Savin \cite{savin} has considered the case of $G$ simply laced and simply connected. The double cover of a general simply connected algebraic group has also been considered by Loke and Savin \cite{lokesavin}. This paper, in developing a theory in greater generality, borrows heavily on the results and arguments from the above mentioned papers, as often the existing arguments in the literature can be generalised to the case of an arbitrary reductive group. In another direction, we mention the work of Weissman \cite{weissman} where the representation theory of metaplectic non-split tori is studied.

This paper is intended to be partly expository, and partly an extension of the work mentioned above, redone to hold in slightly greater generality. We begin by giving an overview of the construction of the metaplectic group. In order to carry out our construction in the desired generality, we are forced to use a significant amount of the theory of these extensions in the semisimple simply connected case, after which we proceed to the general reductive case along the lines of the approach in \cite{finkelberglysenko}.

Our study of the representation theory begins by focussing on the metaplectic torus and its representations, which govern a large part of the following theory. This metaplectic torus is no longer abelian, but its irreducible representations are finite dimensional by a version of the Stone-von Neumann theorem.

Following this, we construct the principal series representations for a metaplectic group by the familiar method of inducing from a Borel subgroup. The theory of Jacquet modules and intertwining operators between such representations is developed in this generality.

We then study the Hecke algebras of anti-genuine compactly supported locally constant functions invariant on the left and the right with respect to an open compact subgroup. The two cases of interest to us are when this compact subgroup is taken to be the maximal compact subgroup $K={\bf G}(O_F)$ or an Iwahori subgroup. In the former case, we present a metaplectic version of the Satake isomorphism in Theorem \ref{satakeiso}, while in the Iwahori case, we give a presentation of the corresponding Hecke algebra in terms of generators and relations, following Savin \cite{savinold,savin}.

With the study of the structure of these Hecke algebras, we propose a combinatorial definition of the dual group to a metaplectic group which extends the notion of a dual group to a reductive group. This dual group is always a reductive group, so unlike in the reductive case, a metaplectic group cannot be recovered from its dual group. We hope that this notion of a dual group will prove to be useful in order to bring the study of metaplectic groups under the umbrella covered by the Langlands functoriality conjectures. It is worth noting that the root datum for the dual group has also appeared in \cite[\S 2.2.5]{lusztigbook}, \cite{finkelberglysenko} and \cite{reich}.

It is believed to be possible to develop this theory while working under the weaker assumption that $F$ only contains $n$ $n$-th roots of unity, though in order to achieve this, a large amount of extra complications in formulae is necessary.

The author would like to thank Ben Brubaker and Omer Offen for useful conversations.

\section{Preliminaries}\label{notation}
As mentioned in the introduction, $F$ will denote a non-archimedean local field, $O_F$ is its valuation ring, and $k$ its residue field. Let $q$ be the cardinality of $k$. We choose once and for all a uniformising element $\p\in O_F$ (so $O_F/\p O_F=k$).

Let $\bf{ G}$ be a split reductive group scheme over $O_F$. Throughout, our practice will be to use boldface letters to denote the group scheme and roman letters to denote the corresponding group of $F$-points. Let ${\bf B}$ be a Borel subgroup of ${\bf G}$ with unipotent radical ${\bf U}$, and let ${\bf T}$ be a maximal split torus contained in ${\bf B}$. We let $\Y=\Hom(\mathbb{G}_m,\TT)$ be the group of cocharacters of ${\bf T}$.

Our object of study is not $G$ itself, but a central extension of $\G$ by the group $\mu_n$ of $n$-th roots of unity. A construction of $\G$ is given in section \ref{2}. For any subgroup $H$ of $G$, we will use the notation $\H$ to denote the inverse image of $H$ in $\G$, it is a central extension of $H$ by $\mu_n$. The topology on $G$ induces the structure of a topological group on $\tilde{G}$. We will use $p$ to denote the natural projection map $\G\rightarrow G$.

We consistently phrase our results in terms of coroots and cocharacters. Given a coroot $\a$, there is an associated morphism of group schemes $\varphi_\alpha\map{\bf SL_2}{{\bf G}}$. Accordingly, we define elements of $G$ by $w_\alpha=\varphi_\alpha((\begin{smallmatrix}
  0 & 1 \\
  -\!1 & 0
\end{smallmatrix}))$ and $e_\a(x)=\varphi_\alpha((\begin{smallmatrix}
  1 & x \\
  0 & 1
\end{smallmatrix}))$. We will show in Section \ref{splitting} that the Weyl group and unipotent subgroups of $G$ split in the central extension $\G$ (in the latter case canonically) and use the same notation for the corresponding lifts to $\G$. For any $x\in F$ and $\la\in\Y$, we will denote the image  in $T$ of $x$ under $\la$ by $x^\la$.

We fix a positive integer $n$ which shall be the degree of our cover. The assumption on $n$ we work under is that $2n|q-1$. This implies the condition that $F$ contains $2n$ $2n$-th roots of unity that was mentioned in the introduction.

We require some knowledge of the existence and properties of the Hilbert symbol, which we shall now recap. These results concerning the Hilbert symbol are well-known, a reference for this material may be given by Serre's book \cite{serre}.

The Hilbert symbol is a bilinear map $(\cdot,\cdot):F^\times \times F^\times\rightarrow \mu_n$ such that
\[
(s,t)(t,s)=1=(t,-t)=(t,1-t).
\]
Due to our assumptions on $n$, we can calculate the Hilbert symbol via the equation
\[
(s,t)=\overline{\left((-1)^{v(s)v(t)}\frac{s^{v(t)}}{t^{v(s)}}\right)}^\frac{q-1}{n}
\]
where the bar indicates to take the image in the residue field and $v$ is the valuation in $F$. Particular special cases that we will make liberal use of throughout this paper without further comment are the identities $(-1,x)=1$ and $(\p^a,\p^b)=1$, both of which require $q-1$ to be divisible by $2n$, as opposed to only being divisible by $n$.

A representation $(\pi,V)$ of $\G$ is a vector space $V$ equipped with a group homomorphism $\pi\map{\G}{\Aut(V)}$. We say that $(\pi,V)$ is \emph{smooth} if the stabiliser of every vector contains an open subgroup, and \emph{admissible} if for every open compact subgroup $K$, the subspace $V^K$ of vectors fixed by $K$ is finite dimensional).

Fix a faithful character $\epsilon\map{\mu_n}{\C^\times}$. We will only have casue to consider representations of $\tilde{G}$ in which the central $\mu_n$ acts by $\epsilon$. (If $\mu_n$ did not act faithfully on an irreducible representation, then this representation would factor through a smaller cover of $G$.) Such representations will be called \emph{genuine}. 

We always will assume our representations to be genuine, smooth and admissible, and denote the category of such representations by $\Rep(\G)$.

For any subgroup $B$ of a group $A$, we use $C_A(B)$ (respectively $N_A(B)$) to denote the centraliser (respectively normaliser) of $B$ in $A$.

\section{Construction of the extension}\label{2}
This section is devoted to showing the existence of the central extension $\G$ that we will be studying.

The cocharacter group $\Y$ comes equipped with a natural action of the Weyl group $W$. Let $B\map{Y\times Y}{\Z}$ be a $W$-invariant symmetric bilinear form on $\Y$ such that $Q(\a):=B(\a,\a)/2\in \Z$ for all coroots $\a$. (There will never be any possible confusion between this use of the symbol $B$ and its use for a Borel subgroup). Associated to such a $B$, we will construct an appropriate central extension.

We begin by recalling the following result of Brylinski and Deligne \cite{bd}.

\begin{proposition}\label{bdexist}
Suppose that ${\bf G}$ is semisimple and simply connected. The category of
central extensions of ${\bf G}$ by ${\bf K}_2$ as sheaves on the big Zariski site $\Spec(F)_{\text Zar}$ is equivalent to the category of integer valued Weyl group invariant quadratic forms on $Y$, where the only morphisms in the latter category are the identity
morphisms.
\end{proposition}

Upon taking $F$-points, this yields a central extension
$$
1\rightarrow K_2(F)\rightarrow E\rightarrow G\rightarrow 1.
$$ Since $F$ is assumed to be a local non-archimedean field containing $n$ $n$-th roots of unity, there is a surjection $K_2(F)\rightarrow\mu_n$ given by the Hilbert symbol and hence we obtain our central extension
$$ 1\rightarrow \mu_n\rightarrow \G\rightarrow G\rightarrow 1.
$$

In particular, if we consider the case where ${\bf G}={\bf SL_2}$, then we have defined a map of abelian groups $\xi\map{\Z}{H^2(SL_2(F),\mu_n)}$.

\begin{theorem}\label{existence}
For any split reductive group $G$, with our assumptions on $F$, $n$ and $B$ as above, there exists a central extension $\G$
of $G$ by $\mu_n$, such that for each coroot $\a$, the pullback of the central extension under $\phi_\a$ to a central extension of $SL_2(F)$ is incarnated by the cohomology class of $\xi(Q(\a))$, and when restricted to a central extension of $T$, the commutator $[\cdot,\cdot]\map{T\times T}{\mu_n}$ is given by
\[
[x^\la,y^\mu]=(x,y)^{B(\la,\mu)}
\] for all $x,y\in F^\times$ and $\la,\mu\in Y$.
\end{theorem}

A more explicit form of the commutator formula is possible, and we discuss it now since it will be of use to us later. Pick a basis $e_1,\ldots,e_r$ of $Y$. This induces an explicit isomorphism $(F^\times)^r\simeq T$, namely $(t_1,\ldots,t_r)\mapsto t_1^{e_1}\ldots t_r^{e_r}$. Via this isomorphism, suppose that $s=(s_1,\ldots,s_r)$ and $t=(t_1,\ldots,t_r)$ are two elements of $T$. Then their commutator is given by the formula
\begin{equation}\label{commutator}
[s,t]=\prod_{i,j}(s_i,t_j)^{b_{ij}}=\prod_i(s_i,\prod_j t_j^{b_{ij}})
\end{equation}
where the integers $b_{ij}$ are defined in terms of the bilinear form $B$ by $$B(\sum_i x_ie_i,\sum_j y_je_j)=\sum_{i,j}b_{ij}x_i y_j.$$

We first discuss Theorem \ref{existence} in the case where ${\bf G}$ is semisimple and simply-connected. In this case, the existence of our desired central extensions of $G$ by $\mu_n$ are well-known to exist. Steinberg \cite{steinberg} computed the universal central extension of $G$ and Matsumoto \cite{matsumoto} computed the kernel of this extension, showing it to be equal to the group $K_2(F)$ except when $G$ is of symplectic type, in which case $K_2(F)$ is canonically a quotient of this kernel. In fact, the adjoint group always acts by conjugation on the universal central extension, and the group of coinvariants is always $K_2(F)$. More recently, Brylinski and Deligne \cite{bd} have generalised this construction, proving the result we quoted above as Proposition \ref{bdexist}.

Theorem \ref{existence} is thus known in this semisimple simply-connected case. The desired commutator relation upon restriction to the torus is proved in \cite[Proposition 3.14]{bd}.

It is most important to us that this central extension is both derived from a solution to a universal problem (so that any automorphism of $G$ lifts to an automorphism of the extension), and that this extension has no automorphisms.

We now discuss Theorem \ref{existence} in the case where ${\bf G}={\bf T}$ is a torus. In this case, the construction of \cite{bd} does not produce all natural central extensions. For example, even in the case of $\bf{GL_1}$ over a field of Laurent series, the work of \cite{kacetal} produces a central extension whose commutator is the Hilbert symbol, whereas $K_2$-based methods only produce the square of this extension. The author does not know how to generalise this construction to the case of mixed characteristic, so we shall proceed in the following ad hoc manner, making use of the assumption that $\mu_{2n}\subset F$. We will write $(\cdot,\cdot)_{2n}$ for the Hilbert symbol with values in $\mu_{2n}$ and reserve $(\cdot,\cdot)$ for the Hilbert symbol with values in $\mu_n$.

We can associate to $2B$ a quadratic form $Q$ and consider the corresponding central extension
\[
1\rightarrow \mu_{2n}\rightarrow E\rightarrow T\rightarrow 1.
\] This extension may be either constructed from the work of \cite{bd}, or we may construct it explicitly from the 2-cocycle
\[
\sigma(s,t)=\prod_{i\leq j}(s_i,t_j)_{2n}^{q_{ij}}
\]
where $Q(\sum_i y_ie_i)=\sum_{i\leq j}q_{ij}y_iy_j$.

This central extension has commutator
\[
[x^\la,y^\mu]=(x,y)_{2n}^{2B(\la,\mu)}=(x,y)^{B(\la,\mu)}.
\]
We will realise $\T$ as an index 2 subgroup of $E$.

Since the commutator takes values only in the subgroup $\mu_n$ of $\mu_{2n}$, when we quotient out by $\mu_n$, we obtain a central extension
\[
1\rightarrow \mu_{2}\rightarrow E'\rightarrow T\rightarrow 1.
\] where the group $E'$ is abelian.

The central extension splits over ${\bf T}(O_F)$ since the Hilbert symbol is trivial on $O_F^\times\times O_F^\times$.

Now choose a splitting of ${\bf T}(O_F)$ and quotient $E'$ by its image. We arrive at a central extension
\[
1\rightarrow \mu_{n}\rightarrow E''\rightarrow T/{\bf T}(O_F)\rightarrow 1.
\] This is a short exact sequence of abelian groups with the last term free. Hence it splits, so we get a surjection $E\rightarrow\mu_2$. Taking the kernel of this surjection as $\T$, we get our desired central extension of $T$ by $\mu_n$.

We now complete the proof of Theorem \ref{existence} for an arbitrary split reductive ${\bf G}$ by a reduction to the semisimple simply connected case and the torus case. This argument follows the proof of \cite[Proposition 1]{finkelberglysenko}. For our purposes here, it will be most convenient for us to reinterpret a central extension of a group $H$ by a group $A$ as a group morphism $H\rightarrow BA$, where $BA$ is the (Milnor) classifying space of $A$. We caution the reader that $f\map{H}{BA}$ is not a homomorphism of topological groups, but instead is only a group homomorphism up to homotopy. For example, the two maps $H\times H\rightarrow BA$ given by $(f\times f)\circ m$ and $m\circ f$ where $m$ is the multiplication, are not equal, but are homotopic.

We briefly indicate how this translation works. By definition, a morphism $H\rightarrow BA$ is an $A$-torsor over $H$. Since $A$ is abelian, the multiplication map from $A\times A$ to $A$ is a group homomorphism and this induces a group structure on $BA$. The extra structure of a group homomorphism on the map $H\rightarrow BA$ is what yields the datum of a group structure on the total space of the torsor represented by this map. It is this group structure on this total space which is the central extension of $H$ by $A$.

Let $G^{sc}$ be the simply connected cover of the derived group of $G$, and let $T^{sc}$ be the inverse image of $T$ in $G^{sc}$. Suppose that we have two central extensions $G^{sc}\rightarrow B\mu_n$ and $T\rightarrow B\mu_n$ which are isomorphic upon restriction to $T^{sc}$. Suppose furthermore that the extension $G^{sc}\rightarrow B\mu_n$ is invariant under the conjugation action of $T$ on $G^{sc}$ and the trivial $T$-action on $\mu_n$. To this set of data, we construct a central extension $G\rightarrow B\mu_n$.

Consider the semidirect product $G^{sc}\rtimes T$ where $T$ is acting on $G^{sc}$ by conjugation. The datum we have of group morphisms $T\rightarrow B\mu_n$ and $G^{sc}\rightarrow B\mu_n$ with the latter being $T$-equivariant is equivalent to that of a group morphism $G^{sc}\rtimes T\rightarrow B\mu_n$.

Now consider the multiplication map from $G^{sc}\rtimes T$ to $G$. This is a surjective group homomorphism with kernel isomorphic to $T^{sc}$ embedded inside $G^{sc}\rtimes T$ via $t\mapsto (t,t^{-1})$.

If we assume that the restrictions of $G^{sc}\rightarrow B\mu_n$ and $T\rightarrow B\mu_n$ to $T^{sc}$ are assumed isomorphic, the restriction of $G^{sc}\rtimes T\rightarrow B\mu_n$ is a trivial central extension. Choosing a trivialisation, we may thus factor $G^{sc}\rtimes T\rightarrow B\mu_n$ through the quotient $G$ of $G^{sc}\rtimes T$ by $T^{sc}$ and thus we get our desired extension $G\rightarrow B\mu_n$.

Now from our knowledge of the semisimple simply-connected case and the torus case, associated to $B$ we may construct our desired central extensions of $G^{sc}$ and $T$ by $\mu_n$. It remains to show that these two extensions agree upon restriction to $T^{sc}$, and that
$G^{sc}\rightarrow B\mu_n$ is invariant under the conjugation action of $T$.

 The former property may be seen by an analogous argument to the one appearing in the proof of the torus case: both extensions split over ${\bf T^{sc}}(O_F)$ and have the same commutator, so their difference in $H^2(T,\mu_n)$ is abelian, splitting over ${\bf T^{sc}}(O_F)$, hence trivial. The latter property holds since the extension by $K_2$ is a canonically constructed object that has no automorphisms. Consequently the action of $T$ on $G^{sc}$ by conjugation extends to an action on the cover $\G^{sc}$ that is trivial when restricted to the central $\mu_n$. This completes our proof of Theorem \ref{existence}.

If we restrict ourselves to the case where $SL_2$, then we can be very explicit about our extension. Following Kubota \cite{kubota}, we have the following formula for $\sigma\in H^2(SL_2(F),\mu_n)$ such that multiplication in $\G$ is given by $(g_1,\zeta_1)(g_2,\zeta_2)=(g_1g_2,\zeta_1\zeta_2\sigma(g_1,g_2))$.

\begin{equation}\label{sl2cocycle}
\sigma(g,h)=\left( \frac{x(gh)}{x(g)} , \frac{x(gh)}{x(h)} \right)^{Q(\a)},
\end{equation}
where for $g=(\begin{smallmatrix}
  a & b \\
  c & d
\end{smallmatrix})\in SL_2(F)$, we define $x(g)=c$ unless $c=0$ in which case $x(g)=d$. In this formula, $\a$ is a simple coroot.


\section{Splitting Properties}\label{splitting}
A subgroup $H$ of $G$ is said to be split by the central extension if we have an isomorphism $p^{-1}(H)\simeq \mu_n\times H$ that commutes with the projection maps to $H$. In this section, we shall show that the unipotent subgroups, and the maximal compact subgroup $K$ are split in $\G$. We also discuss splittings of discrete subgroups corresponding to the coroot lattice and the Weyl group.

\begin{proposition}
Any unipotent subgroup $U$ of $G$ is split canonically by the central extension $\G$.
\end{proposition}
\begin{proof}
This result is proved in greater generality in \cite[Appendix 1]{mw}. Since we are concerned only with the case where $(n,q)=1$, a simple proof can be given. The assumptions on $n$ ensure that the map $U\rightarrow U$ given by $u\mapsto u^n$ is bijective. If $u\in U$, write $u=u_1^n$ and let $\tilde{u_1}$ be any lift of $u_1$ to $\G$. Then define $s(u)=(\tilde{u_1})^n$. This is well defined, invariant under conjugation and is the proposed section determining the splitting.

So it suffices to show that $s$ is a group homomorphism. For $U$ abelian this is trivial. In general, $U$ is solvable, write $U'$ for the quotient group $U/[U,U]$ and by induction, we may assume that $s$ is a homomorphism when restricted to $[U,U]$. We now form the quotient $\widetilde{U}'=\U/s([U,U])$. Suppose $u_1,u_2\in U$. Form $\zeta=s(u_1)s(u_2)s(u_1u_2)^{-1}$. A priori, we have $\zeta\in\mu_n$. Projecting into $\widetilde{U}'$ and using the abelian case of the proposition implies that $\zeta\in s([U,U])$ and thus we have $\zeta=1$ so $s$ is a homomorphism as desired.
\end{proof}

For the corresponding result for the maximal compact subgroup, the splitting is no longer unique, and in order for the splitting to exist, it is essential that $n$ is coprime to the residue characteristic.

\begin{theorem}\cite[\S 10.7]{bd}, \cite[Lemma 11.3]{moore}
 The extension $\G$ of $G$ splits over the maximal compact subgroup $K={\bf G}(O_F)$.
\end{theorem}
\begin{proof}
 Let $K_1$ denote the kernel of the surjection ${\bf G}(O_F)\rightarrow {\bf G}(k)$. Then $K_1$ is a pro-p group, hence has trivial cohomology with coefficients in $\mu_n$. By the Lyndon-Hochschild-Serre spectral sequence, we thus have an isomorphism $H^2(K,\mu_n)\simeq H^2({\bf G}(k),\mu_n)$. Let ${\bf M}$ be the normaliser of ${\bf T}$ in ${\bf G}$. Since the index of ${\bf M}(k)$ in ${\bf G}(k)$ is coprime to $n$, we know that the map $H^2({\bf G}(k),\mu_n)\rightarrow H^2({\bf M}(k),\mu_n)$ is injective. ${\bf T}(k)$ can be considered as the $k$ points of the group scheme of ($q-1$)-th roots of unity in ${\bf T}$, which is etale, so lifts uniquely into ${\bf T}(O_F)$. The group generated by this lift together with the elements $w_\a\in K$ form a lift of ${\bf M}(k)$ into $K$. Thus it suffices to show that our central extension is trivial when restricted to this lift of ${\bf M}(k)$ in $K$. However we have explicit knowledge of the central extension in terms of a 2-cocycle on $M$ thanks to \cite[Lemme 6.5]{matsumoto} (the non-simply connected case works similarly), allowing us to complete the proof in this manner.

\end{proof}
When needed, we will denote by $\kappa^*$ the lifting of $K$ to $\G$. This lifting $\kappa^*$ is not unique, being well defined only up only to a homomorphism from $K$ to $\mu_n$. In \cite{kp} a canonical choice is made in the case of $G=GL_n$, however this failure of uniqueness shall not be of concern to us.

Just as in the case of $G=SL_2$ when we were able to provide an explicit formula for the cocycle $\sigma$, again in this case we are able to provide an explicit formula for the splitting $\kappa$, following Kubota \cite{kubota}. Writing $s$ for the section $g\mapsto (g,1)$, and $\kappa^*$ as $\kappa^*(k)=s(k)\kappa(k)$, we have
\begin{equation}\label{sl2splitting}
\kappa\begin{pmatrix}
  a & b \\
  c & d
\end{pmatrix}=\begin{cases} (c,d)^{Q(\a)} &\text{if  } 0<|c|<1 \\
1 &\text{otherwise,  }  \end{cases}
\end{equation} for all $(\begin{smallmatrix}
  a & b \\
  c & d
\end{smallmatrix})\in SL_2(O_F)$, where again $\a$ is the simple coroot.

The group $Y$, considered as a subgroup of $T$ by the injection $\la\mapsto\p^\la$ is trivially split in our central extension since it is a free abelian group whose cover is abelian (here we require the fact that $\mu_{2n}\subset F$ to conclude that the Hilbert symbols appearing in the commutator vanish).

Let $W_0$ be the subgroup of $G$ generated by all elements of the form $w_\a$ for $\a$ a coroot. This group $W_0$ is a finite cover of the Weyl group $W$. It is split in our extension since it is a subgroup of $K$. Let $W_{a,0}$ be the corresponding cover of the affine Weyl group, that is the group generated by $Y$ and $W_0$. We can see this subgroup of $G$ is also split in our extension by following through the construction of \cite[Lemme 6.5]{matsumoto}, noting that the Weyl group action on $Y$ preserves its splitting. For any $w\in W_{a,0}$, we will identify $w$ with its image under this splitting in $\G$, and we will denote this splitting by $s$ when necessary.

\section{Heisenberg group representations}

A Heisenberg group is defined to be any two-step nilpotent subgroup, ie $H$ is a Heisenberg group if its commutator subgroup is central. The metaplectic torus $\T$ is an example of such a group, since the commutator subgroup $[\T,\T]$ is contained in the central $\mu_n$. The representation theory of Heisenberg groups is well understood, in particular we will make use of the following version of the Stone-von Neumann theorem, compare for example with \cite[Theorem 3.1]{weissman}.

\begin{theorem}\label{heisenbergreps}
 Let $H$ be a Heisenberg group with centre $Z$ such that $H/Z$ is finite and let $\chi$ be a character of $Z$. Suppose that $\ker(\chi)\cap[H,H]=\{1\}$. Then up to isomorphism, there is a unique irreducible representation of $H$ with central character $\chi$. It can be constructed as follows: Let $A$ be a maximal abelian subgroup of $H$ and let $\chi'$ be any extension of $\chi$ to $A$. Then inducing this representation from $A$ to $H$ produces the desired representation.
\end{theorem}

\begin{proof}
 Let $\pi$ be an irreducible $H$-representation with central character $\chi$. Since $H/Z$ is finite, $\pi$ is finite dimensional. Considering $\pi$ as an $A$-representation, this implies that it has a one-dimensional quotient $\chi'$ which must be an extension of the character $\chi$ to $A$.

 By Frobenius reciprocity, there is thus a non-trivial $H$-morphism from $\pi$ to $\Ind_A^H\chi'$. To conclude that $\pi$ is isomorphic to $\Ind_A^H\chi'$, we need to prove that this induced representation is irreducible.

 Since $\Ind_A^H\chi'$ is generated by a non-zero function supported on $A$, to prove irreducibility, it suffices to show that any $H$-invariant subspace contains such a function. So suppose $f\neq 0$ is in a $H$-invariant subspace $M$. Then by translating by an element of $H$, we may assume that the support of $f$ contains $A$. Now suppose that there exists $h\in H\setminus A$ such that $f(h)\neq 0$. Then since $A$ is a maximal abelian subgroup of $H$, there exists $a\in A$ such that $[h,a]\neq 1$. Now consider $(a-\epsilon([a,h])\chi'(a))f\in M$. It has strictly smaller support than $f$ (since it vanishes at $h$) and is non-zero on $A$. So continual application of this method will, by finite dimensionality, produce a non-zero function in $M$ supported on $A$ and thus we've proved the irreducibility of $\Ind_A^H\chi'$.

 To finish the proof of the theorem, we need to show that if we have two different extensions $\chi_1$ and $\chi_2$ of $\chi$ to $A$, then after inducing to $H$, we get isomorphic representations. Given two such extensions, $\chi_1\chi_2^{-1}$ is a character of $A/Z$, and we extend it to a character of $H/Z$. Our assumption that $\ker(\chi)\cap[H,H]=\{1\}$ implies that the pairing $\langle\cdot,\cdot\rangle\map{H/Z\times H/Z}{\C^\times}$ given by $\langle h_1,h_2\rangle=\chi([h_1,h_2])$ is nondegenerate and hence that every character of $H/Z$ is of the form $h\mapsto \chi([h,x])$ for some $x\in H$. Hence there exists $x\in H$ such that $\chi_1\chi_2^{-1}(a)=\chi([a,x])$ for all $a\in A$. This implies that the characters $\chi_1$ and $\chi_2$ conjugate by $x$ under the conjugation action of $H$ on $A$. Hence when induced to representations of $H$, the induced representations are isomorphic.
\end{proof}


\begin{corollary}\label{metaplectictorusreps}
Genuine representations of $\T$ are parametrised by characters of $Z(\T)$.
\end{corollary}
\begin{proof}
 $\T$ is a Heisenberg group so we only need to check that the conditions of the above theorem are satisfied. The condition that $\T/Z(\T)$ is finite is satisfied since $T^n$ (the subgroup of $n$-th powers in $T$) is central and $(F^\times)^n$ is of finite index in $F^\times$.
 The condition that $[H,H]\cap\ker(\chi)=\{1\}$ is satisfied for genuine characters $\chi$, since $[\T,\T]\subset\mu_n$ and $\chi$ is faithful on $\mu_n$. Hence we may apply Theorem \ref{heisenbergreps} to obtain our desired corollary. \end{proof}

We now produce explicitly a choice of maximal abelian subgroup of $\T$ that we can use later on for our convenience.
\begin{lemma}
The group $C_{\T}(\T\cap K)$ is a maximal abelian subgroup of $\T$.
\end{lemma}
\begin{proof}
Since $\T\cap K$ is abelian, it is clear that any maximal abelian subgroup of $\T$ containing $\T\cap K$ is contained in $C_{\T}(\T\cap K)$. So it suffices to prove that $C_{\T}(\T\cap K)$ is abelian.

Recall the basis $e_1,\ldots,e_r$ of $Y$ which was used to introduce coordinates on $T$ and the coefficients $b_{ij}$ of the bilinear form $B$ introduced at the beginning of section \ref{2}. We need to make use of the equation (\ref{commutator}), which we reproduce here for convenience. 

\begin{equation}
[s,t]=\prod_{i,j}(s_i,t_j)^{b_{ij}}=\prod_i(s_i,\prod_j t_j^{b_{ij}}).
\end{equation}

Thus the condition for $t$ to be in $C_{\T}(\T\cap K)$ is that $\prod_j t_j^{b_{ij}}$ has valuation divisible by $n$ for all $i$. Now suppose that $s$ and $t$ are elements of $C_{\T}(\T\cap K)$. Let $x_i$ and $y_i$ be the valuations of $s_i$ and $t_i$ respectively.
Then we have that $(s_i,t_j)$ is equal to $((-1)^{x_iy_j}s_i^{y_j}/t_j^{x_i})^{\frac{q-1}{n}}$ after reduction modulo $\p$. Hence we may compute the commutator
\[
[s,t]=\left(\prod_{i,j}(-1)^{x_iy_j}\prod_i s_i^{\sum_{j}y_jb_{ij}}\prod_j t_j^{\sum_i x_ib_{ij}}\right)^\frac{q-1}{n}.
\]
Since we are assuming that $2n|q-1$, the power of $-1$ which appears in this product is even. By assumption on $s$ and $t$, all exponents of $s_i$ and $t_j$ are divisible by $n$. So the whole product is a $q-1$-th power, so after reduction modulo $\p$ becomes 1. Hence $s$ and $t$ commute, so $C_{\T}(\T\cap K)$ is abelian, as required.
\end{proof}


We will use $H$ to denote this maximal abelian subgroup $C_{\T}(\T\cap K)$.

\section{Principal Series Representations}\label{pcreps}
We begin by studying the class of representations that will be our main object of study. Let $(\pi,V)$ be a genuine, smooth admissible representation of $\T$. The group $\B$ contains $U$ as a normal subgroup with quotient naturally isomorphic to $\T$. Via this quotient, we consider $(\pi,V)$ as a representation of $\B$ on which $U$ acts trivially.

\begin{definition}
For $(\pi,V)$ a smooth representation of $\T$, we define the (normalised) induced representation $I(V)$ as follows:

The space of $I(V)$ is the space of all locally constant functions $f\map{\G}{V}$ such that
\[
f(bg)=\delta^{1/2}(b)\pi(b)f(g)
\] for all $b\in\B$ and $g\in \G$ where $\delta$ is the modular quasicharacter of $\B$ and we are considering $(\pi,V)$ as a representation of $\B$. The action of $\G$ on $I(V)$ is given by right translation. In this way we define an induction functor
\[
I\map{\Rep(\T)}{\Rep(\G)}.
\]
\end{definition}

Suppose now that $\chi$ is a genuine character of $Z(\T)$. We denote by $i(\chi)=(\pi_\chi,V_\chi)$ a representative of the corresponding isomorphism class of irreducible representations of $\T$ with central character $\chi$. By the considerations in the above section, $i(\chi)$ is finite dimensional. We will write $I(\chi)$ for the corresponding induced representation $I(i(\chi))$ of $\G$. Such representations $I(\chi)$ will be called \emph{principal series representations}.


We now define a family of principal series representations, called \emph{unramified} that will be of principal interest to us.

\begin{definition}
A genuine character $\chi$ of $Z(\T)$ is said to be unramified if it has an extension to $H$ that is trivial on $\T\cap K$. We use the same adjective unramified for the corresponding representation $I(\chi)$ of $\G$.
\end{definition}

\begin{lemma}\label{spherical}
An unramified principal series representation $I(\chi)$ has a one-dimensional space of $K$-fixed vectors.
\end{lemma}
\begin{proof}
Suppose $f\in I(\chi)^K$. Let $g=f(1)\in i(\chi)$. By the Iwasawa decomposition $G=UAK$, we may write any $g\in\G$ as $g=uak$ with $u\in U$, $a\in\T$ and $k\in K$. Then we have $f(g)=f(uak)=\pi_\chi(a)g$.

The element $a$ is well defined up to right multiplication by an element $\eta\in \T\cap K$. This induces (the only) compatibility condition, we thus require that $f(g)=\pi_\chi(a\eta)g=\pi_\chi(a)\pi_\chi(\eta)g$. Thus we have that $f\mapsto f(1)$ is an isomorphism from $I(\chi)^K$ to $i(\chi)^{\T\cap K}$.

If $g\in i(\chi)^{\T\cap K}$, then for all $t\in\T$ and $\eta\in {\T\cap K}$ we have
$$
g(t)=g(t\eta)=[t,\eta]g(\eta t)=[t,\eta]\chi(\eta)g(t).
$$
Since $\chi$ is unramified, $\chi(\eta)=1$, so either $[t,\eta]=1$ or $g(t)=0$. The function $g\in i(\chi)$ is determined by its restriction to a set of coset representatives of $H\lqt \T$. By the definition of $H$ we know that $[t,\eta]=1$ for all $\eta\in \T\cap K$ if and only if $t\in H$. Thus we have shown that $i(\chi)^{\T\cap K}$ is one dimensional, proving the lemma.
\end{proof}
A $K$-fixed vector in such a representation will be called \emph{spherical}.

We now define an action of the Weyl group $W$ on principal series representations.

The group $\M$ acts on $\T$ by conjugation, and hence acts on $\Rep(\T)$. Explicitly, write $c_m\map{\T}{\T}$ for the operation of conjugation by $m\in\M$ on $\T$. Then for $(\pi,V)$, the action of $m\in \M$ is defined by $(\pi,V)^m=(\pi^m,V^m)$ where $V^m=V$ and $\pi^m$ is the composition $\pi\circ c_m\map{\T}{\Aut(V)}$.

Unfortunately when we restrict this action to $\T$, we do not obtain the identity. However we may still define an action of the Weyl group on $\Rep(\T)$. Recall from the discussion at the end of section \ref{splitting}, that the group $W_0$ lifts to $\M$. In this reaslisaton of $W_0$, the kernel of the surjection $W_0\rightarrow W$ lies in $Z(\T)$. Since the conjugation action of $Z(\T)$ is trivial, we are able to define an action of $W$ on $\Rep(\T)$ by first restricting the action of $\M$ to an action of $W_0$, which then induces a well-defined action of $W$ on $\Rep(\T)$.

In a similar but simpler manner, one may define an action of $W$ on the space of characters of $Z(\T)$. These two actions are compatible in the sense that $ i(\chi^w)=i(\chi)^w$ for all characters $\chi$ and $w\in W$.

To proceed, we require the theory of the Jacquet functor, and the results regarding the composition of the Jacquet functor with the induction functor $I$. These results are all contained in \cite{bz}.

\begin{definition}
The Jacquet functor $J$ from $\Rep(\G)$ to $\Rep(\T)$ is defined to be the functor of $U$-coinvariants.
\end{definition}

Explicitly, for an object $V$ in $\Rep(\G)$, $J(V)$ is defined to be the largest quotient of $V$ on which $U$ acts trivially, that is we quotient out by the submodule generated by all elements of the form $\pi(u)v-v$ with $u\in U$ and $v\in V$. Since $\T$ normalises $U$ (as $U$ has a unique splitting, conjugation by $\T$ must preserve this splitting), the action of $\T$ on $V$ induces an action on $J(V)$, so the image of $J$ is indeed in $\Rep({\T})$.

The work in \cite{bz} is in sufficient generality to cover our circumstances. In particular, we have the following two propositions.

\begin{proposition}\cite[Proposition 1.9(a)]{bz}
The Jacquet functor is exact.
\end{proposition}

\begin{proposition}\cite[Proposition 1.9(b)]{bz}
The Jacquet functor $J$ is left adjoint to the induction functor $I$. That is, there exists a natural isomorphism
$$
\Hom_{\T}(J(V),W)\cong\Hom_{\G}(V,I(W)).
$$
\end{proposition}

The main result of \cite{bz} is their Theorem 5.2, from which we derive the following important corollary.

\begin{corollary}\label{compositionfactors}
 The composition factors of the Jacquet module $J(I(\chi))$ are given by $i(\chi^w)$ as $w$ runs over $W$.
\end{corollary}

\begin{proof} The derivation of this corollary follows in exactly the same manner as in the reductive case. In the notation of Bernstein and Zelevinsky \cite{bz}, we apply their Theorem 5.2 with ${\bf G}=\G$, ${\bf Q}={\bf P}=\B$, ${\bf N}={\bf M}=\T$, ${\bf V}={\bf U}=U$ and $\psi=\theta=1$.
\end{proof}

We say that a character $\chi$ of $Z(\T)$ is regular if $\chi^w\neq\chi$ for all $w\neq 1$.
\begin{proposition}\label{semisimple}
If $\chi$ is regular, then $J(I(\chi))$ is semisimple.
\end{proposition}
\begin{proof}
Decompose the $\T$-module $J(I(\chi))$ into $Z(\T)$-eigenspaces -- this must be a semisimple decomposition since we are dealing with commuting operators on a finite dimensional space. As $\chi$ is assumed to be regular, these eigenvalues of $Z(\T)$ are all distinct. This shows that the filtration from Corollary \ref{compositionfactors} splits as $\T$-modules, so we're done.
\end{proof}

\section{Intertwining operators}
We start with these results on the spaces of morphisms between various principal series representations.

\begin{theorem}
\begin{enumerate}
\item For two characters $\chi_1$ and $\chi_2$ of $Z(\T)$, we have
$$
Hom_{\G}(I(\chi_1),I(\chi_2))=0
$$ unless there exists $w\in W$ such that $\chi_2=\chi_1^w$.

\item
Suppose that $\chi$ is regular. Then for all $w\in W$ we have
$$
\dim\Hom_{\G}(I(\chi),I(\chi^w))=1
$$
\end{enumerate}
\end{theorem}

\begin{proof}
Since $J$ is left adjoint to $I$, we have $$\Hom_{\G}(I(\chi_1),I(\chi_2))=\Hom_{\T}(J(I(\chi_1)),i(\chi_2)).$$ Our knowledge of the description of the composition series of $J(I(\chi_1))$ from Corollary \ref{compositionfactors} and Proposition \ref{semisimple} completes the proof.
\end{proof}

This section will be dedicated to the explicit construction and analysis of these spaces $\Hom_{\G}(I(\chi),I(\chi^w))$. Elements in these spaces are referred to as intertwining operators.


Suppose $s\in\C$. Associated to $s$ is a one dimensional representation $\delta^s$ of $\B$ given by raising the modular quasicharacter $\delta$ to the $s$-th power. Accordingly, given any representation $V$ of $\T$, we define a family $I_s$ of representations of $\G$ by $$I_s(V)=I(V\otimes \delta^s).$$

For each $V\in \Rep(\T)$, this family of representations is a trivialisable vector bundle over $\C$. We choose a trivialisation  as follows.

To each $f\in I(V)=I_0(V)$ and $s\in C$, we define the element $f_s\in I_s(V)$ by \begin{equation}\label{fsdefn}
f_s(bk)=\delta(b)^s f(bk)
\end{equation} for any $b\in \B$ and $k\in K$. It is easily checked that this is well defined, the claim $f_s\in I_s(V)$ is true and that $s\mapsto f_s$ does define a section.

Our strategy for constructing intertwining operators is as follows. We shall first construct intertwining operators via an integral representation that is only absolutely convergent on a cone in the set of all possible characters $\chi$. We then make use of the trivialising section we have just constructed to meromorphically continue these intertwining operators to all $I(\chi)$.


For any finite dimensional $\T$ representation $(\pi,V)$, and any coroot $\a$, we define the $\a$-radius $r_\a(V)$ to be the maximum absolute value of an eigenvalue of the operator $\pi(\p^\a)$ on $V$. This turns out to be independent of the choice of uniformiser $\p$ since $T(O_F)$ is compact.

For $w\in W$ and such a finite dimensional representation $(\pi,V)$, the intertwining operator $T_w\map{I(V)}{I(V^w)}$ is defined by the integral
\begin{equation}\label{intertwining}
(T_w f)(g)=
\int_{U_w} f(w^{-1}ug)du.
\end{equation} whenever this is absolutely convergent.

To check that $T_w$ does indeed map $I(V)$ into $I(V^w)$ is a simple calculation. Note that the underlying vector spaces of $V$ and $V^w$ are equal as per the definition of the $W$-action on such representations in section \ref{pcreps}.

\begin{lemma}
Suppose that $w_1,w_2\in W$ are such that $\ell(w_1 w_2)=\ell(w_1)+\ell(w_2)$. Then $T_{w_1w_2}=T_{w_1}T_{w_2}$, whenever their defining integrals are absolutely convergent.
\end{lemma}

\begin{proof}
This result is a simple application of Fubini's theorem.
\end{proof}

Let us now restrict ourselves to a study of the case where $w=w_\a$ is the simple reflection corresponding to the simple coroot $\a$.

\begin{theorem}
The defining integral (\ref{intertwining}) for the intertwining operator $T_{w_\a}$ is absolutely convergent for $r_\a(V)<1$.
\end{theorem}

\begin{proof}
In $SL_2$, we have the following identity
$$
\begin{pmatrix}
0 & -1 \\
1 & 0
\end{pmatrix}
\begin{pmatrix}
1 & x \\
0 & 1
\end{pmatrix}=
\begin{pmatrix}
1/x & -1 \\
0 & x
\end{pmatrix}
\begin{pmatrix}
1 & 0 \\
1/x & 1
\end{pmatrix}
$$
We apply the morphism $\phi_\a$ to interpret this as an identity in $G$. This equation lifts to $\G$ as the relevant Kubota cocycles are trivial. We are thus able to write
\begin{eqnarray*}
(T_{w_\a}f)(g)&=&\int_F f(w_\a^{-1}e_\a(x)g)dx\\
&=&\int_F f(e_\a(-1/x)x^\a e_{-\a}(1/x)g)dx \\
&=&\int_F \delta^{1/2}(x^\a)\pi(x^\a)f(e_{-\a}(1/x)g)dx
\end{eqnarray*}

In the above, $e_\gamma(x)$ is the canonical lift from $G$ to $\G$ of the one dimensional unipotent subgroup corresponding to the coroot $\gamma$, as defined in Section \ref{notation}.

Since $f$ is locally constant, there exists a positive number $N$ such that for $|x|\geq N$ we have $f(e_{-\a}(1/x)g)=f(g)$. Now we shall break up our integral over $F$ into a sum of two integrals, the first over $|x|<N$ and the second over $|x|\geq N$. The first integral is an integral over a compact set so is automatically convergent. We will now study the second integral in greater detail.

Note that $\p^{n\a}$ is central in $\T$. We may assume without loss of generality that $f(g)\in V$ is an eigenvector of $\pi(\p^{n\a})$ with corresponding eigenvalue $(q^{-1}x_\a)^n$. Then our second integral becomes
\begin{equation}\label{continue}
\int_{|x|\geq N} (\delta^{1/2}\pi)(x^\a)f(e_{-\a}(1/x)g)dx=\left( \int_{m\leq v(x)<m+n} (\delta^{1/2}\pi)(x^\a)f(g) \right)
\left( \sum_{i=0}^\infty x_\a^{in} \right).
\end{equation}
This is absolutely convergent if and only if $|x_\a|<1$, proving the theorem. \end{proof}


For ease of exposition, we shall now restrict ourselves to the case of intertwining operators from $I(\chi)$ to $I(\chi^w)$. Under this restriction, the complex numbers $x_\a$ are essentially well-defined, in that different choices of eigenvectors will only change them by an $n$-th root of unity. We may pick any such eigenvector to define the $x_\a$, any subsequent formulae will be independent of such choices.

Define a renormalised version of the intertwining operator by
\begin{equation}\label{renormalise}
\T_{w}=\prod_{\substack{\a>0 \\ w\a<0}}(1-x_a^{n}) T_{w}.
\end{equation}

\begin{lemma}
The collection of renormalised intertwining operators $\T_{w_\a}$ satisfies the braid relations.
\end{lemma}

\begin{proof}
Since we know that the unnormalised intertwining operators $T_{w_\a}$ satisfy the braid relations, to check this Lemma, it suffices to check that $c(w_1w_2,x)=c(w_1,w_2x)c(w_2,x)$ where $c(w,x)$ is the renormalising coefficient in (\ref{renormalise}). This is a triviality.
\end{proof}



We are now in a position to analytically continue the intertwining operators $\T_w$.

If $\la$ is the eigenvalue of $\pi(\p^{n\a})$ on $V$, then $\la q^{-s}$ is the eigenvalue of $\pi(\p^{n\a})$ on $V\otimes\delta^s$. Then by (\ref{continue}), $(\T_{w_\a}f_s)(g)$ is a polynomial in $\la q^{-s}$, so in particular is a holomorphic function in $s$. Recall that the section $f_s$ is as defined in (\ref{fsdefn}).

For $\Re(s)$ sufficiently large, the defining integral for $\T_{w_\a}f_s$ is absolutely convergent. Thus we can define $\T_{w_\a}f_s$ for all $s\in \C$ by analytic continuation. In particular, for all $V$, we have now defined

$$\T_{w_\a}\map{I(V)}{I(V^{w_\a})}$$
and since the maps $\T_{w_\a}$ satisfy the braid relations, we have also defined

$$
\T_{w}\map{I(V)}{I(V^{w})}
$$ for all $w\in W$.

Now let us suppose that $V=i(\chi)$ is an irreducible unramified representation of $\T$. By Lemma \ref{spherical}, $I(V)$ contains a $K$-fixed vector. Let $\phi_K$ be such a vector for $I(V)$ and $\phi_K^w$ be such a vector for $I(V^w)$. We normalise these spherical functions such that $(\phi_K(1_{\G}))(1_{\T})=1$. The spherical vectors $\phi_K$ and $\phi_K^w$ are related by $\T_w$ in a manner given by the following theorem. The integer $n_\a$ is defined to be $\frac{n}{(n,Q(\a))}$ where the notation $(\cdot,\cdot)$ here is that of the greatest common divisor.

\begin{theorem}\cite[Theorem 6.5]{decom}\label{gk}
$$\T_w\phi_K=\prod_{\a\in\Phi_w}(1-q^{-1}x_\a^{n_\a})\frac{1-x_\a^n}{1-x_\a^{n_\a}}\phi_K^w.$$
\end{theorem}

\begin{proof}
The proof in \cite{decom} is for the case of $G$ semisimple and simply connected, so we need to show how to reduce to this case. First we note that $T_w\phi_K$ is a priori $K$-fixed, so by Lemma \ref{spherical}, it suffices to calculate the integral
\[
I_\chi=\left(\int_U \phi_K(w^{-1}u)du\right)(1_{\T}).
\]
Consider the natural map from the corresponding simply connected semisimple group $G^{\rm sc}_{\rm ss}$ to $G$. We can pullback the central extension $\G$ of $G$ to a central extension of $G^{\rm sc}_{\rm ss}$ and thus consider the corresponding group $H^{\rm sc}_{\rm ss}$. The character $\chi$ of $H$ can be extended to a character $\chi'$ of $H^{\rm sc}_{\rm ss}$. In calculating $I_\chi$, only group elements in the image of $G^{\rm sc}_{\rm ss}$ occur and we see that the calculation is the same as for the corresponding integral $I_{\chi'}$. In this way, this theorem is reduced to the semisimple, simply connected case.
\end{proof}

\begin{corollary}\label{winvariance}
For generic $\chi$ (so on a Zariski open subset of such characters), the intertwining operator $\T_w$ induces an isomorphism $I(\chi)\simeq I(\chi^w)$.
\end{corollary}

\begin{proof}
The functor $\T_w$ restricts to a morphism from $I(\chi)^K$ to $I(\chi^w)^K$. These two spaces are one-dimensional, so we have an isomorphism as long as $\T_w\phi_K$ is nonzero. The Corollary now follows immediately from Theorem \ref{gk}.
\end{proof}

At this point, we have developed the theory as far as is necessary for the purposes of the Satake isomorphism. Following the works of Casselman \cite{eternalpreprint}, Kazhdan-Patterson \cite{kp} and Rodier \cite{rodier}, one could push this line of though further to produce stronger results on the composition series of principal series representations, though we shall not do this here.

\section{Whittaker Functions}
In this section we consider $(\pi,V)$ a spherical genuine admissible representation of $\G$. Let $\psi$ be a character of $U$ such that the restriction of $\psi$ to each one dimensional subgroup $U_\alpha$ for $\alpha$ a simple coroot is non-trivial.

Let $\mathcal{W}$ denote the space of smooth functions $f\map{\G}{\C}$ such that $f(\zeta ng)=\zeta\psi(n)f(g)$ for $\zeta\in\mu_r$ and $n\in N$. Then a Whittaker model for $(\pi,V)$ is defined to be a $\G$-morphism from $V$ to $\mathcal{W}$. A Whittaker function is any non-zero spherical vector in a Whittaker model. It thus is a function $\ww\map{\G}{\C}$ satisfying
\begin{equation}\label{eqn1}
\ww(\zeta ngk)=\zeta\psi(n)\ww(g)\quad\text{for}\ \zeta\in\mu_r, n\in N, g\in G, k\in K
\end{equation}

Define the twisted Jacquet functor $J_\psi$ from $\Rep(\G)$ to $\hbox{Vect}_\C$ by $J_\psi(V)=V/V_\psi(U)$, where $V_\psi(U)$ is the subspace of $V$ generated by the vectors $\pi(u)v-\psi(u)v$ for all $u\in U$ and $v\in V$. There is a natural bijection between $J_\psi(V)$ and the vector space of Whittaker models of $V$.

Theorem 5.2 of \cite{bz} can be used to compute the dimension of the space of Whittaker functions in the same manner as it was used to compute the composition series of a Jacquet module of an induced representation.

\begin{theorem}
The dimension of the space of Whittaker functions for a principal series representation $I(\chi)$ is $|\T/H|$.
\end{theorem}

We apply \cite[Theorem 5.2]{bz} with ${\bf G}=\G$, ${\bf P}=\B$, ${\bf M}=\T$, ${\bf U}={\bf Q}={\bf V}=U$, ${\bf N}=1$, $\theta=1$ and $\psi$ non-trivial as above. Of the glued functors that appear in the composition series of $J_\psi(I(\chi))$ via \cite[Theorem 5.2]{bz}, only one is non-zero, and it is the forgetful functor from $\Rep(\T)$ to $\hbox{Vect}_\C$.

If $f$ is a spherical vector in $I(\chi)$, then we can construct a Whittaker function as the integral
\[
W(g)=\int_U f(w_0ug)\psi(u)du.
\]
Technically speaking, this is a $i(\chi)$-valued function, so to obtain a $\C$-valued Whittaker function, we should compose with a functional on $i(\chi)$. Such a choice is made in \cite{decom} where a complex-valued Whittaker function is evaluated. In fact, in \cite{decom}, a basis for the space of Whittaker functions is computed together with the production of an explicit formula for $W(t)$ with $t\in\T$ in the case where $G=SL_n$. Note that by (\ref{eqn1}) and the Iwasawa decomposition, $W$ is completely determined by its restriction to $\T$. There is an alternative method of Chinta and Offen \cite{chintaoffen} for calculating these metaplectic Whittaker functions. Their method more closely follows the lines of the original work of Casselman and Shalika \cite{cs}, again working in type A.

\section{The spherical Hecke algebra}
We call a complex valued function $f$ on $\G$ anti-genuine if, for all $\zeta\in\mu_n$ and $g\in\G$, we have $f(\zeta g)=\zeta^{-1}f(g)$. This notion is of use to us since we are only studying genuine representations of $\G$. If we decompose the algebra $C_c^\infty(\G)$ of smooth compactly supported functions on $\G$ into a direct sum of eigenspaces under the action of $\mu_n$, then only the anti-genuine functions act non-trivially on a genuine representation of $\G$. We now define and study a version of the spherical Hecke algebra for the metaplectic group.

Considering $K$ as a subgroup of $\G$ via $\kappa^*$, let $\h$ denote the algebra of $K$-bi-invariant anti-genuine compactly supported smooth (locally constant) complex valued functions. In other words, a compactly supported smooth function $f\map{\G}{\C}$ is in $\h$ if and only if $f(\zeta k_1 gk_2)=\zeta^{-1} f(g)$ for all $\zeta\in\mu_r$, $g\in G$ and $k_1,k_2\in K$. The algebra structure is given by convolution, for $f_1,f_2\in \h$, we define
$$
(f_1f_2)(g)=\int_{\G}f_1(h)f_2(h^{-1}g)dh
$$
where the Haar measure on $\G$ is normalised such that $K\times\mu_n$ has measure 1.

We have the following two results about the structure of $\h$. In the case of $G=GL_n$, these appear in \cite{kp86}.

\begin{theorem}\label{kthm}
$\h$ is commutative.
\end{theorem}
We will not prove this in this section, but instead note that it follows immediately from the Satake isomorphism, Theorem \ref{satakeiso}.

\begin{theorem}\label{ksupport}
 The support of $\h$ is given by $\mu_n KHK$.
\end{theorem}

\begin{proof}
The Cartan decomposition $G=KTK$ implies that every $(K,K)$ double coset of $G$ contains a representative of the form $\p^\la$, and this decomposition clearly lifts to $\G$. So it suffices to find the set of $\la$ for which the double coset $\mu_n K\p^\la K$ supports a function in $\h$.

Fix $\la$, and let $K^\la$ denote the subgroup $K\cap \p^{-\la}K\p^\la$ of $G$. We define a function $\phi^\la\map{K^\la}{\mu_n}$ as follows. For $k\in K^\la$ there exists a unique $k'\in K$ such that $k\p^\la=\p^\la k'$. We lift this identity into $\G$ using our choice of splitting of $K$, and define $\phi^\la(k)$ by $k\p^\la=\phi(k)\p^\la k'$.

It is straightforward to check that $\phi^\la$ is a group homomorphism. Furthermore, there is a function in $\h$ supported on $\mu_n K\p^\la K$ if and only if the homomorphism $\phi^\la$ is trivial.

 The normal subgroup $K_1\cap K^\la$ of $K^\la$ is a pro-$p$ group, hence the homomorphism $\phi^\la$ is trivial when restricted to this subgroup.

There is a canonicial isomorphism $K^\la/(K_1\cap K^\la)\simeq {\bf P}(k)$ for some parabolic subgroup ${\bf P}$ of ${\bf G}$. The above shows that $\phi^\la$ factors to a homomorphism from ${\bf P}(k)$ to $\mu_n$. The group ${\bf P}(k)$ is generated by ${\bf T}(k)$ and unipotent elements. Since $\phi^\la$ is necessarily trivial on any unipotent element, it is completely determined by its restriction to ${\bf T}(k)$.

We know that the restriction of $\phi^\la$ to ${\bf T}(O_F)$ is trivial if and only if $\p^\la\in H$, by the definition of $H$. This completes our proof. \end{proof}

\section{Satake Isomorphism}\label{satake}

The approach we shall take in presenting the Satake isomorphism was learnt by the author from a lecture of Kazhdan in the reductive case, and differs from that which is generally considered as for example in \cite{gross}. First we define a free abelian group $\La$ which shall be of fundamental importance for the remainder of this paper. Let $$\La=\{\lambda\in \Y\mid s(\p^\lambda)\in H\}=\{ x\in \Y\mid B(x,y)\in n\Z\ \forall\ y\in Y \}.$$ The equivalence of the two given presentations is a consequence of the commutator formula (\ref{commutator}). This group $\La$ is also naturally isomorphic to the abelian group $H/(\T\cap K\times \mu_n)$, and carries an action of the Weyl group, inherited from the action of $W$ on $T$.

The aim of this section is to prove the following.
\begin{theorem}[Satake Isomorphism]\label{satakeiso}
Let $\C[\La]$ denote the group algebra of $\La$. Then there is a natural isomorphism between the spherical Hecke algebra $\h$ and the $W$-invariant subalgebra, $\C[\La]^W$.
\end{theorem}

 Let $Z_\Lambda$ denote the complex affine variety $\Hom(\Lambda,\C^\times)$, and $\Gamma_\Lambda$ be the ring of regular functions on $Z_\Lambda$. We shall first define a homomorphism from $\h$ to $\Gamma_\Lambda$.

To any $\chi\in Z_\Lambda$ there is an associated genuine unramified principal series representation $I(\chi)=(\pi_\chi,V_\chi)$ of $\G$. By Lemma \ref{spherical}, this representation has the property that $\dim V_\chi^K=1$, and thus $V_\chi^K$ is a one-dimensional representation of $\h$. We again use $\pi_\chi \map{\h}{\End(V_\chi^K)\cong\C}$ to denote this representation.

From this representation, we obtain a ring homomorphism
$
S\map{\h}{\Gamma_\Lambda}
$
given by $Sf(\chi)=\pi_\chi(f)$. This is the Satake map. A priori, the image of this map lies in the set of functions from $Z_\la$ to $\C$, though it will follow from the results proven below that the image lies in the ring of regular functions on $Z_\la$.

For any abelian group $\Lambda$ there is a canonical isomorphism between $\Gamma_\Lambda$ and the group ring of $\Lambda$ (which is actually the same as given above, if we can take $\G=\Lambda$ in the definition of the Satake map).

Let us identify $\Gamma_\Lambda$ with $\C[\Lambda]$ via this isomorphism. Using this we will from now assume that $S$ has image in $\C[\Lambda]$.

\begin{lemma}\label{satakeintegrallemma}
We have the following formula for the Satake map $S\map{\h}{\C[\La]}$
\begin{equation}\label{integralsatake}
(Sf)(\lambda)=\delta^{1/2}(\p^\lambda)\int_U f(\p^\lambda u) du.
\end{equation}
\end{lemma}

\begin{proof}
We begin by unfolding of the integral definition of the action of $f$ on the spherical vector $\phi_K$. From this we get
\begin{eqnarray*}
\pi_\chi(f)\phi_K
&=&\int_{\G} f(g)\pi_\chi(g)\phi_K dg \\
&=& \int_K\int_{\B} f(bk)\pi_\chi(bk)\phi_K d_L\!b\ dk \\
&=&\int_{\B} f(b)\pi_\chi(b)\phi_K d_L\!b \\
&=&\int_{\T}\int_U f(tu)\pi_\chi(tu)\phi_K du\ dt \\
&=& \int_{\T} \left(\delta^{1/2}(t)\int_U f(tu)du\right)(\delta^{-1/2}\pi_\chi)(t)\phi_K dt.
\end{eqnarray*}
It was shown in the proof of Lemma \ref{spherical} that for $t\in \T$, $(\delta^{-1/2}\pi_\chi)(t)\phi_K=\chi(t)\phi_K$ if $t\in H$ and is zero otherwise. Thus we may restrict our integral over $\T$ to an integral over $H$. Since the integrand is invariant under $\T\cap K\times \mu_n$, we obtain the following sum over $\La$
\[
\pi_\chi(f)=\sum_{\la\in\La} \delta^{1/2}(\p^\la) \int_U f(\p^\la u)du \chi(\p^\la).
\]
Under the isomorphism $\Gamma_\La\simeq \C[\La]$, this gives us (\ref{integralsatake}) as required.
\end{proof}

\begin{lemma}
The image of the Satake map lies in $\Gamma_\Lambda^W$.
\end{lemma}
\begin{proof}
By Corollary \ref{winvariance}, we have, for generic $\chi$, an isomorphism between $I(\chi)^K$ and $I(\chi^w)^K$.
 Thus the image of the Satake map is $W$-invariant. To complete the proof, it remains to show that the image of $S$ consists of regular functions on $Z_\la$ (or equivalently that $(Sf)(\la)$ is nonzero for only finitely many $\la$). For this we use the integral expression from Lemma \ref{satakeintegrallemma}. To see this, we need to remark that any $f\in\h$ is compactly supported, and use \cite[Proposition 4.4.4(i)]{bruhattits}.
\end{proof}

\begin{theorem}
The Satake map $S$ gives an isomorphism between $\h$ and $\C[\Lambda]^W$.
\end{theorem}

\begin{proof}
For dominant $\la\in\La$, we define basis elements $c_\la$ and $d_\la$ of $\h$ and $\C[\La]^W$ respectively.

Let $c_\la$ be the function in $\h$ that is supported on $\mu_n K\p^\la K$ and takes the value 1 at $s(\p^\la)$. That the set of all such $c_\la$ form a basis of $\h$ is known from Theorem \ref{ksupport}.

Let $d_\la\in \C[\La]^W$ be the characteristic function of the orbit $W\la$.

Write $Sc_\la=\sum_{\mu}a_{\la\mu}d_\mu$. We shall show that $a_{\la\la}\neq 0$ and that $a_{\la\mu}=0$ unless $\mu\leq\la$, which suffices to prove that $S$ is bijective. Since we already know that $S$ is a homomorphism, this is sufficient to prove our theorem.

To show that $a_{\la\la}\neq 0$, we must calculate $Sc_\la(\la)$. Notice that for $u\in U$, we have $\p^{\la} u\in K\p^\la K$ if and only if $u\in K$ so in the calculation of the integral (\ref{integralsatake}), the integrand is non-zero only on $K\cap U$, where it takes the value 1, hence the integral is non-zero, so $a_{\la\la}\neq 0$ as desired.

To show that $a_{\la\mu}=0$ unless $\mu\leq\la$, we again look at calculating $Sc_\la(\mu)$ via the integral (\ref{integralsatake}). We again appeal to a result from the structure theory of reductive groups over local fields \cite[Proposition 4.4.4(i)]{bruhattits} to say that $\p^\mu U\cap K\p^\la K=0$ unless $\mu\leq\la$, which immediately gives us our desired vanishing result, so we are done.

\end{proof}

\section{The dual group to a metaplectic group}

Motivated by the Satake isomorphism in the previous section, we will now give a combinatorial definition of a dual group to a metaplectic group. This group $\G^\vee$ will be a split reductive group, so to define it, it will suffice to give a root datum $(X,\Phi,X',\Phi')$.

We use $\Delta$ to denote the set of all coroots. Throughout this section, lower case Greek letters will be used to denote coroots. If $\a$ is a simple coroot, recall that the integer $n_\a$ is defined to be the quotient $n_\a=\frac{n}{(n,Q(\a))}$ (where $(\cdot,\cdot)$ here is used to denote the greatest common divisor).

We define a root datum $(X,\Phi,X',\Phi')$ by
\begin{eqnarray*}
X&=&\La, \\
\Phi\ &=&\{n_\a\a\mid\a\in\Delta\}, \\
X'&=&\Hom(\La,\Z)\subset \Hom({\bf T},\mathbb{G}_m)\otimes\Q, \\
\Phi' &=&\{n_\a^{-1}\a^\vee\mid\a\in\Delta\},
\end{eqnarray*}
and we define the dual group $\G^\vee$ of $\G$ to be the reductive group associated to this root datum.

\begin{theorem}
The quadruple $(X,\Phi,X',\Phi')$ defines a root datum.
\end{theorem}

\begin{proof}
To check that $\Phi$ and $\Phi'$ are stable under the Weyl group is straightforward. For example, if $w\a=\beta$, then $Q(\a)=Q(\beta)$ so $wn_\a\a=n_\beta\beta$. The only part involving significant work is to check that $\Phi\subset X$ and $\Phi'\subset X'$.

To check that $\Phi\subset X$, it suffices to show that for all $\a\in\Delta$ and $y\in \Y$ we have that $B(\a,y)$ is divisible by $Q(\a)$.

Consider the set $L_y=y+\Z\a$. It is a $w_\a$ stable subset of $\Y$. There are two possibilities, either $L_y$ contains $z$ which is fixed by $w_\a$ or $L_\a$ contains $z$ such that $w_\a z=z+\a$.

In the former case, consider the $\Q$-subspace of $Y\otimes \Q$ spanned by $z$ and $\a$. On this subspace we have $Q(m\a+nz)=Am^2+Bn^2+Cmn$ for some $A,B,C\in \Q$. Since $Q$ is invariant under $w_\a$, we must have that $C=0$. Then $B(\a,z)=0$, so since $Q(\a)$ divides $B(\a,\a)$, it must divide $B(\a,y)$.

In the latter case, we calculate that $B(\a,z)=-Q(\a)$, so proceed as in the former case, so we are done.


We now show that $\Phi'\subset X'$.

Firstly we use the fact that $B(\a,\a)=2Q(\a)$ to conclude that
$$
n_\a\Z\a\subset X\cap\Q\a\subset\frac{n_\a}{2}\Z\a.
$$

Now consider some $\beta\in X$. and let $M_\beta=(\beta+\Q\a)\cap X$. A priori, there are three options.

The first is that there exists $\gamma\in M_\beta$ such that $w_\a \gamma=\gamma$, which implies $\langle\a^\vee,\gamma\rangle=0$.

The second is that there exists $\gamma\in M_\beta$ such that $w_\a \gamma=\gamma+n_\a\a$ which implies $\langle\a^\vee,\gamma\rangle=n_\a$.

In the third potential case we would have $\gamma\in M_\beta$ such that $w_\a \gamma=\gamma+\frac{n_\a}{2}\a$. For this to occur, we would require that $2|n_\a$, so in this case $B(\gamma,\a)\notin n\Z$. However this last statement implies that $\gamma\notin X$, which cannot occur.

Thus since we know that $\beta=\gamma+\frac{kn_\a}{2}\a$ for some integer $k$, we obtain that $\langle\beta,\a^\vee\rangle\in n_\a\Z$. This shows that $n_\a^{-1}\a^\vee\in \Hom(X,\Z)=X'$, as required.

\end{proof}

Thus we have a root datum, so defining $\G^\vee$ as the split reductive group corresponding to this root datum is well-defined.

As a consequence, we may consider the Satake isomorphism to be the existence of a natural isomorphism
$$
\h\cong\C[\La]^W \cong K_0(\text{Rep}(\G^\vee))\otimes\C.
$$


\section{Iwahori Hecke Algebra}\label{iwahorisec}

There is an alternative Hecke algebra associated to the group $\G$, defined in the same fashion as the spherical Hecke algebra $\h$, but considering a standard Iwahori subgroup $I$ (defined to be the inverse image of ${\bf B}(k)$ under the surjection $K\rightarrow {\bf G}(k)$) in place of the hyperspecial maximal compact subgroup $K$.
We will denote this Hecke algebra by $\hh$, it is the algebra of antigenuine $I$-biinvariant compactly supported locally constant functions on $\G$.

Let $J$ denote the normaliser in $\G$ of $\T\cap K$.

\begin{theorem}
The support of the algebra $\hh$ is $IJI$.
\end{theorem}
\begin{proof}
We also use the decomposition $G=IMI$. Suppose that $t\in M$ and $t\notin J$. Then there exists $k\in \T\cap K$ such that $tkt^{-1}\notin \T\cap K$. Since we are assuming $t\in WT$, we have that $p(t)\in T\cap K$. Thus $tkt^{-1}=\zeta k'$ for some $k'\in\T\cap K$ and $\zeta\in\mu_n$ with $\zeta\neq 1$. Hence any $f\in\hh$ has $f(t)=0$ so we have proved that the support of $\hh$ lies in $IJI$.

For the reverse implication we need to show that if $t\in J$ then there exists $f\in\hh$ with $f(t)\neq 0$. To do this we need to show that whenever $p(i_1ti_2)=p(t)$ for $i_1,i_2\in I$, then $i_1ti_2=t$.

Let $I_p$ denote the maximal pro-$p$ subgroup of $I$ (it is the inverse image of ${\bf U}(k)$ under the projection $K\rightarrow{\bf G}(k)$). The torus ${\bf T}(k)$ over the residue field lifts to $I$ and every element of $I$ can be uniquely written as a product of an element of ${\bf T}(k)$ with an element of $I_p$.

After projection to $G$, $i_2\in I\cap t^{-1}It$. Write $i_2=j_1j_2$ with $j_1\in {\bf T}(k)$ and $j_2\in I_p$. Then $ti_2t^{-1}=tj_1t^{-1}tj_2t^{-1}$. We have $tj_1t^{-1}\in {\bf T}(k)$ because $t$ normalises $\T\cap K$ and ${\bf T}(k)$ consists of all elements of order $q-1$ in this group. Since $tj_2t^{-1}$ topologically generates a pro-$p$ group, it must be that $tj_2t^{-1}\in I_p$ since it is a priori in $\tilde{I}$ which also has a unique maximal pro-$p$ subgroup. Thus $ti_2t^{-1}=tj_1t^{-1}tj_2t^{-1}\in I$ so we are done.
\end{proof}

Let $W_a$ denote inverse image of $\La$ under the projection from the affine Weyl group to $Q$. Then there is an isomorphism $I\lqt IJI/I\simeq W_a$. As a corollary of the above theorem, we are able to exhibit a basis for $\hh$. For any $w\in W$, we are able to exhibit a choice of a lifting $w\in \G$ which is an element of our embedding $W_0\hookrightarrow \G$ from the discussion at the end of section \ref{splitting}. There is an embedding $\La\subset W_a$ and $W\subset W_a$. Using these inclusions, we identify elements of $\la$ as elements of $W_a$ and for each simple coroot $\a$ denote by $s_\a\in W_a$ the corresponding simple reflection.

\begin{corollary}
For each $w\in W_a$, then there is a function $T_w$ in $\hh$, supported on $\mu_n IwI$ and taking the value 1 at $w$. Then the collection of these $T_w$ for $w\in W_a$, forms a $\C$ basis for the algebra $\hh$.
\end{corollary}

It is possible to write down a system of generators and relations for the algebra $\hh$. The following is a corrected version of \cite[Proposition 3.1.2]{savinold}. The change is in the definition of Savin's integer $m$, which has been replaced by $n_\a$ (although $m=n_\a$ in a large number of cases, in general they are not even equal in the rank one case).

Let $\La^+$ denote the set of dominant elements of $\La$ and $\Delta$ denote the set of simple coroots.

\begin{theorem}\label{iwahorialg}
The following relations hold in $\hh$.
\begin{enumerate}
\item $T_\la T_\mu=T_{\la+\mu}$ for $\la,\mu\in \La^+$.

\item If $s_\a\la=\la$ for $\a\in\Delta$ and $\la\in\La^+$ then $T_{s_\a}$ and $T_{\la}$ commute.

\item If $\la\in \La^+$ and $\langle\a^\vee,\la\rangle=n_\a$ then
$$
T_\la T_{s_\a}^{-1}T_\la T_{s_\a}^{-1}=q^{n_\a-1}T_{2\la-n_\a\a}.
$$
\item If $\la\in \La^+$ and $\langle\a,\la\rangle=2n_\a$ then
$$
T_\la T_{s_\a}^{-1}T_\la T_{s_\a}^{-1}=q^{2n_\a-1}T_{2\la-2n_\a \a}+
(q-1)q^{n_\a-1}T_{2\la-n_\a\a}T_{s_\a}^{-1}.
$$

\item $(T_{s_\a}-q)(T_{s_\a}+1)=0$ for $\a\in \Delta$.

\item For $w_1,w_2\in W$ with $\ell(w_1w_2)=\ell(w_1)+\ell(w_2)$ we have $T_{w_1w_2}=T_{w_1}T_{w_2}$.
\end{enumerate}
\end{theorem}

\begin{proof}
The proof given by Savin \cite{savinold} is applicable here and correct until we reduce to a rank one calculation in proving parts 3 and 4. We will present this rank one calculation here. It does not appear in \cite{savinold} and is the source of the inaccurate statement in \cite[Proposition 3.1.2]{savinold}. To carry out this computation, we will be making use of the explicit formulae for the 2-cocyle and the splitting given in equations (\ref{sl2cocycle}) and (\ref{sl2splitting}) respectively.

In the rank one case, the statements of parts 3 and 4 simplify to the following, where $s$ is the sole reflection in the Weyl group.
\begin{enumerate}
\item[(3')] If $\la=n_\a\a/2\in\La^+$ then $$ T_\la T_{s}^{-1}T_\la =q^{n_\a-1}T_s $$
\item[(4')] If $\la=n_\a\a\in\La^+$ then $$ T_\la T_{s}^{-1}T_\la =q^{2n_\a-1}T_s+
(q-1)q^{n_\a-1}T_{\la}.
 $$
\end{enumerate}

We know from Savin's proof that $T_{s_i}^{-1}T_\la=T_{s_i\la}$ and so need to calculate the product $T_\la T_{s\la}$. In particular, we need to calculate $T_\la T_{s\la}(\p^\la)$, which is the task we shall accomplish.

Let us first consider the case where our rank one group is $G=SL_2$. We may thus write $\p^\la=(\begin{smallmatrix}
  \p^l & 0 \\
  0 & \p^{-l}
\end{smallmatrix})$ for some integer $l$ (actually $l=\langle\a^\vee,\la\rangle/2$). Note that $2lQ(\a)$ is divisible by $n$, which will have the consequence that all powers of Hilbert symbols that appear will be $\pm1$, a feature we will exploit, simplifying our expressions by freely inverting such symbols on a whim.

We write $T_\la T_{s\la}(\p^\la)$ as an integral over $\G/\mu_n\simeq G$.
$$
T_\la T_{s\la}(\p^\la)=\int_G T_\la(h)T_{s\la}(h^{-1}\p^\la)dh.
$$
This integrand is non-zero when $h^{-1}\in I\p^{-\la} I\cap Is\p^{-\la}I\p^{-\la}$. We shall work modulo $I$ on the left. Thus we have
$$
h^{-1}=\begin{pmatrix}
  0 & -\p^{-l} \\
  \p^l & 0
\end{pmatrix}\begin{pmatrix}
  a & b \\
  c & d
\end{pmatrix}\begin{pmatrix}
  \p^{-l} & 0 \\
  0 & \p^{l}
\end{pmatrix}
$$ where $(\begin{smallmatrix}
  a & b \\
  c & d
\end{smallmatrix})\in I$. The condition for $h^{-1}\in I\p^{-\la}I$ is equivalent to $c=\p^l u$ for some $u\in O_F^\times.$

We calculate $h=(\begin{smallmatrix}
  b\p^{2l} & d \\
  -a & -u\p^{-l}
\end{smallmatrix})$ and $s(h)^{-1}=s(h^{-1})$.

For $i_1=(\begin{smallmatrix}
  -u & -d\p^l \\
  0 & -u^{-1}
\end{smallmatrix})$ and $i_2=(\begin{smallmatrix}
  1 & 0 \\
  -au^{-1}\p^l & 1
\end{smallmatrix})$ we have $i_1 h i_2=\p^\la$, $\kappa(i_1)=\kappa(i_2)=1$ and $\sigma(i_1h,i_2)\sigma(i_1,h)=(au,\p^{lQ(\a)})$.

For $i_3=(\begin{smallmatrix}
  d & -b \\
  -u\p^l & a
\end{smallmatrix})$, we have $h^{-1}\p^\la i_3=s\p^\la$, $\kappa(i_3)=(a,\p^\la)$ and $\sigma(h^{-1},\p^\la)\sigma(h^{-1}\p^\la,i_3)=(a,\p^{lQ(\a)})$.

Thus overall our integrand $T_\la(h)T_{s\la}(h^{-1}\p^\la)$ is equal to $(au,\p^{lQ(\a)})$ where it is supported. Hence the integral $T_\la T_{s\la}(\p^\la)$ is equal zero if $n$ does not divide $lQ(\a)$ and the value of the appropriate volume, namely $q^{2n_\a-1}$, otherwise. To complete the proof in the $SL_2$ case, we need to note that in case 3, $2l=n_\a$ and thus $n$ does not divide $lQ(\a)$ as can be seen by looking at $2$-adic valuations. In case 4, $l=n_\a$ so $n$ trivially divides $lQ(\a)$.

Now we turn to the case of $G=PGL_2$. We write $\p^\la=(\begin{smallmatrix}
  \p^l & 0 \\
  0 & 1
\end{smallmatrix})$ In order to have our integrand $T_\la(h)T_{s\la}(h^{-1}\p^\la)$ non-zero, by consideration of the valuation of the determinant, we must have that $l$ is even. This immediately proves our result when $l$ is odd. For in case 4, for $PGL_2$ we have $l=2n_\a$, so if $l$ is odd, we must be in case 3.

If $\la\in \La^+$ is such that $\langle\a^\vee,\la\rangle=2n_\a$ and $\la/2\in\La^+$, then the equation (4) is a formal consequence of (3) and (5). Thus we may reduce to the case where $\la$ is a minimal non-zero element of $\La^+$. This implies that $l$ divides $n$, so in particular, $n$ is even.

Let $E$ be an unramified quadratic extension of $F$. Consider the natural map $SL_2(E)\rightarrow PGL_2(E)$ and restrict this to the preimage of $PGL_2(F)$. Note that $\p^\la$ and all elements of the Iwahori subgroup $I$ lie in the image of this map. Accordingly, we will be able to make use of the above calculation for $SL_2(E)$.

Since $n$ is even, $\frac{q+1}{2}\equiv 1\pmod  n$. For $s,t\in E$ with $s^2,t^2\in F$, we thus have the following identity of Hilbert symbols.
$$
(s^2,t^2)_F^{Q(\a/2)}=\overline{\left((-1)^{v(s)v(t)}\frac{s^{v(t)}}{t^{v(s)}}\right)}^{\frac{q-1}{n}Q(\a)}
=\overline{\left((-1)^{v(s)v(t)}\frac{s^{v(t)}}{t^{v(s)}}\right)}^{\frac{q^2-1}{n}\frac{Q(\a)}{2}}=(s,t)_E^{Q(\a)/2}.
$$
We interpret this in the following manner. Since our central extensions are determined by their restriction to maximal tori, this shows that the pullback of the extension of $PGL_2(F)$ to its inverse image in $SL_2(E)$ is the same as the restriction of the central extension on $SL_2(E)$ corresponding to the quadratic form $Q'$ defined by $Q'=Q/2$. We are able to push forward a cocycle on $SL_2$ to a cocycle on $PGL_2$ since the centre of $SL_2$ remains central when lifted to $\widetilde{SL_2}$.

As a result of this relationship between the covers of $PGL_2(F)$ and $SL_2(E)$, we are able to use the $SL_2$ calculations above for the proof in the $PGL_2$ case. We have $\langle\a^\vee,\la\rangle=l$ and $T_\la(h)T_{s\la}(h^{-1}\p^\la)$ is non-zero if and only if $n$ divides $lQ(\a)/2$.

In case 3, $l=n_\a$. Since $n$ is known to be even, it does not divide $lQ(\a)/2$ by the same $2$-adic argument as in the $SL_2$ case.

In case 4, $l=2n_\a$ and in this case $n$ trivially divides $lQ(\a)/2$.

This completes our calculation, and so combined with the work in \cite{savinold}, completes the proof.
\end{proof}

There is a stronger statement, giving a presentation for the Hecke algebra $\hh$.

\begin{theorem}\cite{savinold}\label{hipres}
The set of relations presented in Theorem \ref{iwahorialg} provides a complete set of relations for the algebra $\hh$.
\end{theorem}
\begin{proof}The proof of Savin \cite{savinold} of this theorem goes through without change.\end{proof}

\section{Further work with the dual group}


\begin{corollary}[\cite{savinold}]
Suppose that $\G$ and $\H$ are two metaplectic groups with isomorphic dual groups $\G^\vee\cong\H^\vee$ and Iwahori subgroups $I^{\G}$ and $I^{\H}$ respectively. Then there is an isomorphism of Iwahori-Hecke algebras
\[
\mathcal{H}_\epsilon(\G,I^{\G}) \cong \mathcal{H}_\epsilon(\H,I^{\H}).
\]
\end{corollary}
\begin{proof}
This is an immediate consequence of the description of these Hecke algebras in terms of generators and relations in Theorem \ref{hipres}. To see this explicitly, we rewrite the relations without any occurrences of $n_\a$ in the exponents of $q$ by defining new variables $U_s=T_s$ and $U_\la=q^{-\langle\rho^\vee,\la\rangle}T_\la$.
\end{proof}

To the data of a metaplectic cover of a split group (that is the group $G$, the quadratic form $Q$ and the degree of the cover $n$), let us propose to define the $L$-group of $\G$ to be the complex reductive group $\G^\vee(\C)$. We hope that this definition will provide a way to bring the study of the metaplectic groups into the paradigm that is the Langlands functoriality conjectures.


The above corollary together with the metaplectic Satake isomorphism provide a starting point for correspondences between local representations with an Iwahori-fixed or spherical vector respectively. In the spherical case, we have the following.

\begin{proposition}
Suppose $\G$ and $\H$ are two metaplectic (possibly reductive) groups with a continuous homomorphism $^L\G\rightarrow\, ^L\H$. Then there is a natural correspondence from spherical representations of $\G$ to spherical representations of $\H$.
\end{proposition}

\begin{proof} The homomorphism $^L\G\rightarrow\, ^L\H$ defines a functor $\Rep(^LH)\rightarrow\Rep(^LG)$. Taking Grothendieck groups and using the Satake isomorphism we obtain a natural morphism of spherical Hecke algebras $\mathcal{H}(\H,K)\rightarrow \mathcal{H}(\G,K)$, hence a map between representations of these spherical Hecke algebras, and thus a correspondence of representations from spherical representations of $\G$ to spherical representations of $\H$.
\end{proof}

We end this paper with a short discussion of a categorified version of the metaplectic Satake isomorphism due to Finkelberg and Lysenko \cite{finkelberglysenko}. Suppose that $F$ is a field of Laurent series $F=k((t))$ over a field $k$ with some mild assumption on the characteristic of $k$ not being too small. Corresponding to $\G$, there is a central extension of the loop group ${\bf G}(F)$ by $\mathbb{G}_m(k)$ as group ind-schemes over $k$. This central extension splits over ${\bf G}(O_F)$, so we obtain a $\mathbb{G}_m$ torsor over the affine Grassmannian $Gr={\bf G}(F)/{\bf G}(O_F)$ (as an ind-scheme over $k$). The group $K={\bf G}(O_F)$ acts on the total space of this torsor $E^\circ$ by left multiplication and the group $\mu_n$ acts by multiplication fibrewise. Again choose a faithful character $\epsilon$ of $\mu_n(k)$. Consider $\epsilon$ as a representation of $\pi_1(\mathbb{G}_m)$ and let $L^\epsilon$ be the corresponding one dimensional local system on $\mathbb{G}_m$. One considers the category of perverse sheaves on $E^\circ$ which are $K$ and $(\mathbb{G}_m,L^\epsilon)$-equivariant. Finkelberg and Lysenko give this category the structure of a tensor category and show that it is equivalent to the category of representations of a reductive algebraic group. They construct explicitly the root system of this group and it can be seen to be the same as the root system constructed above for the group we denoted $\G^\vee$. 


\end{document}